\documentclass[12pt]{amsart}

\usepackage{amsmath}
\usepackage{amsthm}
\usepackage{amssymb}
\usepackage{amsfonts}
\usepackage{bm}
\usepackage[shortlabels]{enumitem}
\usepackage[bookmarks=true,hyperindex,pdftex,colorlinks,citecolor=red, linkcolor=blue]{hyperref}
\usepackage{centernot} 
\usepackage{marginnote} 
\usepackage[left=3cm,right=3cm,top=2cm,bottom=2cm]{geometry}
\usepackage{bbm}
\usepackage[all]{xy}
\usepackage{tikz}
\usepackage{comment}
\usetikzlibrary{arrows}
\usetikzlibrary{shapes,decorations}
\usetikzlibrary{positioning}
\usepackage{xcolor}
\usepackage{bigints}
\usepackage{enumitem}

\theoremstyle{plain}
\newtheorem{thm}{Theorem}[]
\newtheorem{prop}[thm]{Proposition}
\newtheorem{lemma}[thm]{Lemma}

\newtheorem{remark}[thm]{Remark}

\theoremstyle{definition}

\newcommand{\R}{\mathbb{R}}
\newcommand{\T}{\mathbb{T}}
\newcommand{\C}{\mathbb{C}}

\newcommand{\norm}[1]{\left\Vert #1\right\Vert}

\definecolor{darkgreen}{rgb}{.2, .6, .2}

\newcommand{\Tt}[0]{$(T_t)_{t\ge0}$}

\title{Behavior of Kreiss bounded $C_0$-semigroups on a Hilbert space}
\author{L. Arnold}
\date{}

\begin{document}
	\address[L. Arnold]{Jana i Jedrzeja Śniadeckich 8, 00-656 Warszawa, Poland}
\email{larnold@impan.pl}

\begin{abstract}
	Let $0<\alpha \leq 1$. We prove that a $\alpha$-Kreiss bounded $C_0$ semigroup $(T_t)_{t \ge 0}$ on a Hilbert space has asymptotics $\norm{T_t} = \mathcal{O}\big(t^{\alpha}/\sqrt{log(t)}\big)$. Then, we give an application to perturbed wave equation.
\end{abstract}

\subjclass[2020]{Primary 47D06 ; Secondary 34D05, 35B45}

\keywords{$C_0$-semigroup, Kreiss condition, perturbed wave equation}

	\maketitle

	Let $-A$ be the generator of a $C_0$-semigroup $(T_t)_{t\geq 0}$ on a complex Banach space $X$. For $\alpha > 0$, we say that $(T_t)_{t\ge 0}$ is $\alpha$-Kreiss bounded (or simply Kreiss bounded when $\alpha = 1$) if $\sigma(A) \subset \overline{\C_+} = \{ \lambda \in \C: \, \operatorname{Re}(\lambda) \geq 0\}$ and
		\begin{equation}\label{Kreiss}
		\norm{R(\lambda,A)} \leq \frac{C}{(-\operatorname{Re}(\lambda))^{\alpha}}, \quad \forall \operatorname{Re}(\lambda)<0.
		\end{equation}
	According to \cite[(2.6)]{hel-sjo} (see also \cite[Theorem 1.1.(1)]{roz-ver1}) , if \Tt\, is an $\alpha$-Kreiss bounded
	 $C_0$-semigroup on a Hilbert space for $\alpha >0$, then $\norm{T_t} = O(t^{\alpha})$. Let us fix $\gamma \in (0,1)$. In \cite{eis-zwa1}, the authors give an example
	  (see \cite[Example 4.4.]{eis-zwa1}) of Kreiss bounded $C_0$-semigroup for which there
	   exists a constant $C>0$ such that $\norm{T_t}\geq Ct^{\gamma}$ for $t\ge 1$. A natural
	    question (listed as an open question in \cite[Appendix B]{roz1}) is whether the
	     estimate  $\norm{T_t} = O(t^{\alpha})$, for
	    $\alpha$-Kreiss bounded $C_0$-semigroup \Tt\,on a Hilbert space, is sharp. In the following we answer this question for $\alpha \in (0,1]$.

\begin{thm}\label{mainTh}
	Let $0<\alpha\leq 1$ and let $-A$ be the generator of $\alpha$-Kreiss bounded $C_0$-semigroup  $(T_t)_{t\geq0 }$ on a Hilbert space.
	Then 
	\begin{equation}\label{ThLogEstimation}
	\norm{T_t} = \underset{t \rightarrow \infty}{\mathcal{O}}\Big(t^{\alpha}/\sqrt{log(t)}\Big).
\end{equation}
\end{thm}

The proof of the Theorem \ref{mainTh} is based on techniques given in \cite{coh-cun-eis-lin}. We first state and prove two lemmas. It is known (see \cite{gom} or \cite{shifeng}) that a $C_0$-semigroup $(T_t)_{t\ge 0}$ on a Hilbert space, with  negative generator $A$, is bounded if and only if 
\begin{equation*}
	\norm{R(-r + i\cdot,A)x}^2_{L^2(\mathbb{R},H)} \leq \frac{C}{r}\norm{x}^2, \quad r>0
\end{equation*}
and 
\begin{equation*}
	\norm{R(-r + i\cdot,A^*)x^*}^2_{L^2(\mathbb{R},H)} \leq \frac{C}{r}\norm{x^*}^2, \quad r>0.
\end{equation*}

 The following lemma gives an analogous necessary condition for a $C_0$-semigroup to be $\alpha$-Kreiss bounded, for $\alpha >0$. 

\begin{lemma}\label{lemestreskreiss}
	Let $\alpha >0$ and $(T_t)_{t\geq 0 }$ be a $C_0$-semigroup on a Hilbert space whose negative generator $A$. Assume $(T_t)_{t\geq 0}$ is $\alpha$-Kreiss bounded then 
	\begin{equation}\label{wgfs2}
		\norm{R(-r + i\cdot,A)x}^2_{L^2(\mathbb{R},H)} \leq \frac{C(1+r^{\alpha})^2}{r^{2\alpha}}\norm{x}^2, \quad r > 0
	\end{equation}
	and 
	\begin{equation}\label{wgfs2*}
		\norm{R(-r + i\cdot,A^*)x^*}^2_{L^2(\mathbb{R},H)} \leq \frac{C(1+r^{\alpha})^2}{r^{2\alpha}}\norm{x^*}, \quad r > 0.
	\end{equation}
\end{lemma}

\begin{proof}
	 Let $x \in H$. First, we by \cite[formula (7.1) section 1.7]{pazy}, we have
	 		\begin{equation*}
	R(-r +i\beta,A)x = -\int_{0}^{\infty} e^{i \beta s }e^{-r s}T_{s}xds, \quad r > 0, \beta \in \mathbb{R}.
		\end{equation*}
	This means that for $r>0$,
		\begin{equation}\label{ResFou}
	R(-r +i\beta,A)x = -\mathcal{F}^{-1}(s \mapsto e^{-r s}T_sx)(\beta),
	\end{equation}
	where $\mathcal{F} : L^2(\R,H) \mapsto L^2(\R,H)$ is the Fourier-Plancherel operator. By Fourier-Plancherel Theorem
		\begin{equation}\label{FourPlan}
	\norm{\mathcal{F}^{-1}(s \mapsto e^{-r s}T_sx)(\cdot) }_{L^2(\mathbb{R}, H)} = \frac{1}{\sqrt{2\pi}} \norm{e^{-r \cdot}T_{(\cdot)}x}_{L^2(\mathbb{R}_+, H)}. 
	\end{equation}
	The Gearhart-Prüss Theorem (see \cite[Theorem 5.2.1.]{bat-ar}) says that for each
	 $r>0$, $\underset{t>0}\sup\norm{e^{-rt}T_t} < \infty$. This implies that there exists
	  $K>0$ such that $$\underset{a>1}\sup\norm{e^{-a \cdot}T_{(\cdot)}x}_{L^2(\mathbb{R}_+, H)} \le K\norm{x}$$ and then 
	\begin{equation}\label{eg1Lemkrei}
		\underset{a>1}\sup\norm{	R(-a +i\beta,A)x}_{L^2(\mathbb{R}_+, H)} \le \frac{K}{\sqrt{2\pi}}\norm{x}.
		\end{equation}

	Now, let $r>0$ and $a = r+1>1$. By the resolvent identity, we have, for $\beta \in \R$,
	\begin{align*}
	R(-r +i\beta,A)x &= \big(I+(r-a)R(-r +i\beta,A)\big) R(-a+i\beta,A)x \\ 
	&= \big(I-R(-r +i\beta,A)\big) R(-a+i\beta,A)x.
	\end{align*} 
	Using \eqref{Kreiss} and \eqref{eg1Lemkrei}:
	\begin{align*}
		\norm{R(-r + i\cdot,A)x}^2_{L^2(\mathbb{R},H)} & \le(1+\sup_{\beta \in \R}\norm{R(-r+i\beta,A)})^2	\underset{a'<-1}\sup\norm{	R(-a' +i\cdot,A)x}^2_{L^2(\mathbb{R}_+, H)} \\
		&\leq 2\pi K^2\Big(1+\frac{C}{r^{\alpha}}\Big)^2\norm{x}^2 \\
		&\leq C'\frac{(1+r^{\alpha}  )^2}{r^{2\alpha}}\norm{x}^2 . 
	\end{align*}

	In the same way we obtain \eqref{wgfs2*}.
	
\end{proof}

The next lemma says that if \Tt\,and $(T^*_t)_{t\ge0}$ satisfy a kind of Cesàro condition then $\norm{T_t}$ satisfy \eqref{ThLogEstimation}.
\begin{lemma}\label{Lemfinal}
	Let $\alpha \in (0,1)$ and $(T_t)_{t\ge 0}$ be a $C_0$-semigroup on a Banach space $X$. Assume that there exists $C>0$ such that for each $t>1$,
	\[
\int_{0}^t \norm{T_s x}^2 ds \leq Ct^{2\alpha} \norm{x}^2, \quad x \in X,
\]
and 
\[
\int_{0}^t \norm{T^*_s x^*}^{2} ds \leq Ct^{2\alpha} \norm{x^*}^{2}, \quad x^* \in X^* .
\]	
Then, $\norm{T_t} = \mathcal{O}(t^{\alpha}/\sqrt{log(t)})$. 
\end{lemma}
\begin{proof}
		Let $t >2$ and $1<P<Q<t$. We have
	\begin{align*}
		(Q-P)^2|\langle T_tx,x^* \rangle |^2 &=   \Big|\int_{P}^Q \langle T_{t-s}x,T_s^*x^* \rangle ds \Big|^2\\
		& \leq \Big(\int_{P}^Q \norm{T_{t-s}x}^2ds\Big)\Big(\int_{P}^Q \norm{T_{s}^*x^*}^{2}ds\Big) \\
		&  \leq \Big(\int_{P}^Q \norm{T_{t-s}x}^2ds\Big)\Big(\int_{0}^Q \norm{T_{s}^*x^*}^2ds\Big) \\
		& \leq CQ^{2\alpha}\norm{x^*}^2\Big(\int_{t-Q}^{t-P} \norm{T_{s}x}^2ds\Big) 
	\end{align*}
	Finally taking the supremum over $\{\norm{x^*} =1 \}$ one obtains, since $\alpha \in (0,1]$,
	\begin{align}\label{esRk}
	\frac{(Q-P)^2}{Q^2} \norm{T_tx}^2 & \leq CQ^{2(\alpha-1)}\int_{t-Q}^{t-P} \norm{T_{s}x}^2ds \\
	 &\leq C\int_{t-Q}^{t-P} \norm{T_{s}x}^2ds \nonumber
	\end{align}
	Now we set $L:= \left\lfloor log(t)/log(2) \right\rfloor$. For $0 \leq l \leq L-1$,  $Q=2^{l+1}$ and $P= 2^{l}$ we have
	\[
	\norm{T_tx}^2 \leq 4C \int_{t-2^{l+1}}^{t-2^l} \norm{T_sx}^2ds
	\]
	Therefore 
	\[
	L\norm{T_tx}^2 = \sum_{l=0}^{L-1} \norm{T_tx}^2 \leq \sum_{l=0}^{L-1} 4C  \int_{t-2^{l+1}}^{t-2^l} \norm{T_sx}^pds \leq 4C\int_{0}^t \norm{T_sx}^2 ds \leq 4C^2t^{2\alpha}  
	\]
	Hence for $t>2$, 
	\[
	\norm{T_tx} \leq \frac{2Ct^{\alpha}}{\sqrt{L}}\norm{x} \leq \frac{C't^{\alpha}}{\sqrt{log(t)}}\norm{x}.
	\]
	 
\end{proof}

We are now ready to prove Theorem \ref{mainTh}.

\begin{proof}[Proof of Theorem \ref{mainTh}]
	Let $x\in H$ and $r>0$. By \eqref{ResFou}, \eqref{FourPlan} and Lemma \ref{lemestreskreiss}, 
	\[
	\norm{e^{-r \cdot}T_{(\cdot)}x}_{L^2(\mathbb{R}_+, H)}^2 = \norm{R(-r + i\cdot,A)x}^2_{L^2(\mathbb{R},H)} \leq  \frac{C(1+r^{\alpha})^2}{r^{2\alpha}}\norm{x}^2.
	\]
	Furthermore 
	\[
	\norm{e^{-r \cdot}T_{(\cdot)}x}^2_{L^2(\mathbb{R}_+, H)} = \int_{\mathbb{R}}e^{-2r s}\norm{T_sx}^2ds \geq e^{-2r t}\int_{0}^t \norm{T_sx}^2ds.
	\]
	Taking $t = \frac{1}{r}$ we obtain
	\[
	\int_{0}^t \norm{T_s x}^2 ds \leq Ce^{2}t^{2\alpha}\big(1+\frac{1}{t^{\alpha}}\big)^2 \norm{x}^2.
	\]
	Hence, for $t>1$, 
	\[
	\int_{0}^t \norm{T_s x}^2 ds \leq C't^{2\alpha} \norm{x}^2,
	\]
	where $C'=4e^2C$.
	Similarly, since $(T_t^*)_{t\ge 0}$ is Kreiss bounded $C_0$-semigroup on a Hilbert space, for $t>1$ and $x^*\in H$, we have
	\[
	\int_{0}^t \norm{T^*_s x^*}^2 ds \leq C't^{2\alpha} \norm{x^*}^2.
	\]
	Finally, Lemma \ref{Lemfinal} allows us to conclude that $\|T_t\| = \mathcal{O}(t^{\alpha}/\sqrt{log(t)})$. 

\end{proof}

\begin{remark}\label{rk}
	\begin{enumerate} 
		\item When $\alpha >1$, \eqref{esRk} with $Q=2$ and $P=1$ gives 
		$$
		\norm{T_tx}^2 \le C2^{\alpha}\int_{t-2}^{t-1}\norm{T_sx}^2ds \le C'\int_{0}^{t}\norm{T_sx}^2ds \le C''t^{2\alpha}\norm{x}^2.
		$$ 
		Then $\norm{T_t} = O(t^{\alpha})$, which gives another proof of \cite[Theorem 1.1.(1) with $g(s) = s^{\alpha}$  for $\alpha >1$]{roz-ver1}.
		
		\item Let us notice that if \Tt\,satisfies the $\alpha$-Kreiss bounded condition only for $\operatorname{Re}(\lambda) \in (0,1)$ that is 
		\begin{equation*}
			\norm{R(\lambda,A)} \leq \frac{C}{-\operatorname{Re}(\lambda)^{\alpha}}, \quad \forall \operatorname{Re}(\lambda) \in (-1,0),
		\end{equation*}
		then the conclusion $\norm{T_t} = \mathcal{O}(t^{\alpha}/\sqrt{log(t)})$ remains true.
	\end{enumerate}
\end{remark}
Let give an application of the Theorem \ref{mainTh}. Let $W^{2,s}(\T^2)$ be the second order Sobolev space equipped with the standard norm. Let $A$ be the operator with domain $D(A) = W^{2,2}(\T^2) \times W^{2,1}(\T^2)$ defined by 
$$
A = \begin{pmatrix}
	0 & -1 \\
	-\Delta - M\frac{\partial}{\partial x} & 0 
\end{pmatrix},
$$
where $\Delta$ is the Laplacian with $D(\Delta) = W^{2,2}(\T^2)$ and $M :L^2(\T^2) \rightarrow L^2(\T^2) $ is the multiplication operator given by $M(h)(x,y) = e^{iy}h(x,y)$. 
Then, it is known (see \cite[section 4]{roz-ver1} for details) that $-A$ generates a
 $C_0$-group $(T_t)_{t\in \R}$ on the Hilbert space $ H = W^{1,2}(\T^2)\times L^2(\T^2)$,
  satisfying $\norm{T_t} = \underset{t \rightarrow \pm \infty}{\mathcal{O}}(|t|e^{|t|/2})$. With Theorem
   \ref{mainTh}, we are able to improve this estimate:
\begin{prop}
Let $(T_t)_{t \in \R}$ be the $C_0$-group on $H$ generated by $-A$. Then
\begin{equation}\label{estifinal}
	\norm{T_t} = \underset{t \rightarrow \pm \infty}{\mathcal{O}}\Bigg(\frac{|t|}{\sqrt{log(|t|)}}e^{|t|/2}\Bigg).
\end{equation}

\end{prop}	
\begin{proof}
	According to \cite[Lemma 4.4.]{roz-ver1}, for $\lambda \in \C_+$, with $\operatorname{Re}(\lambda) \in (0,1)$,
	$$
	\norm{R\Big(\lambda, A + \frac{1}{2}\Big) } \le \frac{C}{-\operatorname{Re}(\lambda)} \quad {and}  \quad \norm{R\Big(\lambda, -A + \frac{1}{2}\Big) }  \le \frac{C}{-\operatorname{Re}(\lambda)}.
	$$
	Then, according to Remark \ref{rk} \textit{(2)}, the semigroup $(e^{-1/2t}T_t)_{t\ge 0}$ (resp. $(e^{-1/2t}T_{-t})_{t\ge 0}$ ) which is generated by $-(A+\frac{1}{2})$ (resp. $A+\frac{1}{2}$), satisfies $\norm{e^{-1/2t}T_t} = \mathcal{O}(t/\sqrt{log(t)})$ (resp.  $\norm{e^{-1/2t}T_{-t}} = \mathcal{O}(t/\sqrt{log(t)})$). This yields the estimate \eqref{estifinal}.

\end{proof}

\vskip 1cm
\noindent
{\bf Acknowledgements.} Author was supported by the ERC
grant {\it Rigidity of groups and higher index theory} under the European Union’s
Horizon 2020 research and innovation program (grant agreement no. 677120-INDEX).
	\bibliographystyle{plain}
		\bibliography{article-bib}
\end{document}